\documentclass[10pt, a4paper]{article}

\usepackage[USenglish]{babel}

\usepackage[boxed, lined, linesnumbered]{algorithm2e} \usepackage{amsmath, amssymb, amsfonts, amsthm}
\usepackage{array}
\usepackage{color}
\usepackage{dsfont}
\usepackage{enumerate}
\usepackage{enumitem}
\usepackage{epsfig}
\usepackage[T1]{fontenc}
\usepackage{geometry}
\usepackage{graphpap}
\usepackage{graphics}
\usepackage{graphicx}
\usepackage{ifpdf}
\usepackage[latin1]{inputenc}
\usepackage{latexsym}
\usepackage{lscape}
\usepackage{mathrsfs}
\usepackage{pgfplots, pgfplotstable}
\usepackage{placeins}
\usepackage{subfigure}      
\usepackage{textcomp}
\usepackage{tikz}   
\usetikzlibrary{matrix, shapes}
\usepackage{times}                    
\usepackage{wrapfig}

\usepackage{mathtools}
\mathtoolsset{showonlyrefs}

\usepackage{pdfpages} 
\usepackage{pgfplots}
\pgfplotsset{compat=newest}

\usepackage{hyperref}

\hypersetup{
 hypertexnames = false,
 plainpages = false
}

\usepackage{varwidth}
\newcommand\Umbruch[2][3cm]{\begin{varwidth}{#1}\centering#2\end{varwidth}}

\newcommand{\N}{\mathbb{N}}

\newcommand{\R}{\mathbb{R}}

\DeclareMathOperator*{\argmin}{argmin}

\renewcommand{\d}{{\,\mathrm{d}}}

\renewcommand{\div}{{\mathrm{div}}}

\newcommand{\domain}{{D}}

\newcommand{\Leb}{\mathscr{L}}
\newcommand{\prox}{{\mathrm{prox}}}
\newcommand{\projection}{{\mathrm{proj}}}
\newcommand{\indicatorFct}{{\mathbf{1}}}

\newcommand{\weaklyStar}{\stackrel{*}{\rightharpoonup}}

\theoremstyle{definition}
\newtheorem{Definition}{Definition}[section]

\theoremstyle{remark}
\newtheorem{Remark}[Definition]{Remark}
\theoremstyle{plain}
\newtheorem{Proposition}[Definition]{Proposition}

\newtheorem{Theorem}[Definition]{Theorem}

\def\bfnu{{\nu}} 
\def\dd{\;\!\mathrm{d}}
       
\newcommand{\Meas}{\mathscr{M}^+(D)}        \newcommand{\SMeas}{\mathscr{M}(D)}           \newcommand{\SMeasd}{\mathscr{M}(D;\R^d)}
\newcommand{\STMeas}{\mathscr{M}([0,1] \times D)}   \newcommand{\STMeasp}{\mathscr{M^+}([0,1] \times D)}    \newcommand{\STMeasd}{\mathscr{M}([0,1] \times D;\R^d)}
\newcommand{\CE}{\mathcal{CE}}
\newcommand{\cA}{\mathcal{D}}
\newcommand{\ABB}{\cA_{\text{trans}}}
\newcommand{\AM}{\cA_{\text{source}}}

\makeatletter
\DeclareRobustCommand\onedot{\futurelet\@let@token\@onedot}
\def\@onedot{\ifx\@let@token.\else.\null\fi\xspace}

\def\ie{\emph{i.e}\onedot}

\newcommand{\penaltyPushforward}{\delta}
    
\renewcommand{\u}{{\theta}} \newcommand{\E}{{\mathcal{E}}} \newcommand{\Ed}{{\mathcal{E}_\delta}}

\newcommand{\Wd}{{\mathcal{W}_{\delta}}}
\newcommand{\BBfunction}{\Phi}
\renewcommand{\div}{{\mathrm{div}}}
\newcommand{\ling}{r} \newcommand{\Ling}{{\mathcal{R}_h}} 
  
\usepackage{makeidx} \makeindex

\date{}
\title{Generalized optimal transport with singular sources}
\author{Jan Maas\footnote{IST Austria, 3400 Klosterneuburg (Austria)}
, Martin Rumpf and Stefan Simon\footnote{Institut f\"ur Numerische Simulation, Universit\"at Bonn, 53115 Bonn (Germany)}}

\begin{document}

\maketitle
\begin{abstract} 
We present a generalized optimal transport model in
which the mass-preserving constraint for the $L^2$-Wasserstein distance is relaxed by introducing a source term in the continuity equation.
The source term is also incorporated in the path energy by means of its squared 
$L^2$-norm in time of a functional with linear growth in space.
This extension of the original transport model enables local  density modulation, which is a desirable feature in applications such as image warping and blending.
A key advantage of the use of a functional with linear growth in space is that it allows for singular sources and sinks, which can be supported on points or lines.
On a technical level, the $L^2$-norm in time ensures a disintegration of the source in time, which we use to obtain the well-posedness of the model and the existence of geodesic paths. 
Furthermore, a numerical scheme based on the proximal splitting approach \cite{papadakis2014optimal} is presented.
We compare our model with the corresponding model involving the $L^2(L^2)$-norm of the source, which merges the metamorphosis approach and the optimal transport approaches in imaging \cite{RuSc12}.
Selected numerical test cases show strikingly different behaviour. 

\end{abstract}

\section{An optimal transport model with source term}\label{sec:geodesics}
In the last decade optimal transport became a very popular  tool in image processing and image analysis \cite{PaPe14}, where
the quadratic Wasserstein distance is applied for instance in non-rigid image registration and warping. 
It was also used to robustly measure distances between images or
to segment and classify images  \cite{PeFaRa12}. 
Driven by applications for instance in imaging \cite{rubner2000earth,1109.49027,peyre2012wasserstein,burger2012regularized} there is a strong demand 
to develop robust and efficient algorithms to compute optimal transport geodesics, such as the recently presented entropic regularization \cite{carlier2015convergence,peyre2015entropic} or the 
sparse multiscale approach \cite{schmitzer2015sparse}.
In this paper we propose an extension of the optimal transport model which enlarges the range of applications.

In their groundbreaking paper Benamou and Brenier~\cite{BeBr00} reformulated (for numerical purposes) the problem of optimal transport 
first considered by Monge and then relaxed by Kantorovich 
in a continuum mechanical framework describing the evolution of the mass distribution in time. 
This reformulation turned out to be the geodesic equation on the $L^2$ Wasserstein space.
For an underlying flow of a density $\u$ with Eulerian velocity $v$ one considers the path energy
\begin{align}\label{eq:BB}
  \E[\u, v] = \int_0^1 \int_\domain \u |v|^2 \d x \d t \, ,
\end{align}
where $\domain \subset \R^d$ is assumed to be a bounded convex domain with Lipschitz boundary.
Then the quadratic Wasserstein distance between two probability density function $\u_A$ and $\u_B$ can be computed by minimizing $\E$ over all density functions $\u: [0,1] \times \domain \to \R$ and velocity fields $v: [0,1] \times \domain \to \R^d$ subject to the continuity equation $\partial_t \u + \div (\u v) = 0$ and the constraints $\u(0) = \u_A$ and $\u(1) = \u_B$.
Here the continuity equation enforces $\u(t)$ to remain in the space of probability densities.
In applications such as image registration or image morphing, input images are frequently not of the same mass. 
Thus, a contrast modulation on the input images is required before an optimal match between the input images can be computed.
But, even if the total mass of the input images coincides, the incorporation of local intensity modulation is desirable to cope with the variability of natural images and to avoid 
''artificial'' long range transport just for the purpose of mass redistribution between totally independent image structures.

To this end, following \cite{RuSc12}, we relax the mass preserving condition and introduce a source term $z:[0,1] \times \domain \to \R$ in the continuity equation:
\begin{align}\label{eq:source}
  \partial_t \u + \div (\u v) = z \,.
\end{align}
This source terms in then also incoporated in the path energy via a suitable penalty term, which represented the cost 
of mass production. It turns out that a suitable choice of the penalty is the squared $L^2$-norm in time of a functional of $z$ with linear growth in space.
In explicit, we ask for minimizers of the generalized action functional
\begin{align}\label{eq:EnergyWithSourceL2L1}
 \Ed[\u, v, z ] = \int_0^1 \int_\domain \u |v|^2 \d x \d t + \frac{1}{\penaltyPushforward} \int_0^1 \left( \int_\domain \ling(z) \d x \right)^2 \d t \, .
\end{align}
subject to the relaxed continuity equation \eqref{eq:source} and the constraints $\u(0) = \u_A$ and $\u(1) = \u_B$.
Here, $\ling : \R \to \R$ is a non-negative, convex function with linear growth, satisfying $\ling(0) = 0$. 
In this paper we assume that $\ling(z)$ has the same growth as $|z|$.
Cases of interest in our considerations are $r(s) = |s|$ corresponding to the $L^1$ norm in space or a Huber norm in space with
$r(s) = \tfrac1{2\beta} s^2$ for $|s| \leq \beta$ and $|s|-\tfrac\beta2$ else for some $\beta >0$.
Here, $\penaltyPushforward > 0$ denotes a penalty parameter allowing to grade the mass modulation rate.
It is desirable to allow also for singular sources which are supported on line segments of points in space. 
This seems to be impossible with a penalty involving the squared $L^2$ norm of the source term in space and time, as has been suggested previously in \cite{RuSc12}.
The linear growth property will ensure that singular source terms are admissible.
Our model is related to the Hellinger-Kantorovich metric considered in \cite{LiMi15,CPSV15,KMV15}, in which a reaction leads to generation or absorption of mass which is incorporated in the continuity equation.

To establish existence of geodesic paths we work in  the framework of Radon measures and consider a suitable decomposition of the corresponding measures for mass, momentum and source term into absolutely continuous and orthogonal parts with respect to the Lebesgue measure.
Since these decompositions are not unique, it is useful to require $1$-homogeneity of the integrands for the singular measures, which ensures that the definition of the energy functionals does not depend on the decomposition.
Indeed, we will observe that the $L^1$-norm of the source  in space allows singular sources.
Furthermore, the $L^2$-norm in time provides an equi-integrability estimate, which guarantees compactness in the space of curves of Radon measures.

The flow formulation \eqref{eq:BB} has been used in \cite{BeBr00} primarily to compute optimal transport geodesics numerically with an augmented Lagrangian approach.
In \cite{papadakis2014optimal} it was shown that a proximal splitting algorithm leads to an equivalent optimization method.

This paper is organized as follows:
First, in Section \ref{sec:existence} we rigorously define the generalized optimal transport model and establish the existence of optimal transport plans.
Then we propose in Section \ref{sec:numeric} an efficient numerical scheme via proper adaptation of the proximal splitting method. 
Finally,  in Section \ref{sec:numericalResults} we present results and discuss properties of the generalized model.

\section{Existence of geodesics for generalized transport metrics}\label{sec:existence}
We follow the lines of \cite{DNS09}  and propose a measure-valued setup for the energy in \eqref{eq:EnergyWithSourceL2L1} as well as for the continuity equation with source term \eqref{eq:source}.
The corresponding minimization problem will allow us to define a generalized Wasserstein distance on the space of positive Radon measures.
We will focus here on the treatment of the source term based on the squared 
$L^2$-norm in time of a functional with linear growth in space, adapting some arguments from the  $L^2(L^2)$-case presented in \cite{RuSc12}. For the convenience of the reader, we will keep the exposition self contained.

First, we apply the change of variables $(\u,v) \mapsto (\u, m = \u v)$ already used by Benamou and Brenier \cite{BeBr00}.
Instead of the pair $(\u,v)$ we consider the pair $(\u,m)$, where $m$ denotes the momentum, 
such that the integrand $ |v|^2 \u$ pointwise  transforms into
\begin{align*}
 \BBfunction(\u, m ) 
 = \left\{
    \begin{array}{cl} 
       \frac{|m|^2}{\u} & \mbox{if } \u > 0 \, , \\
       0 & \mbox{if } (\u,m) = 0 \, ,\\
       + \infty & \mbox{otherwise} \, .
    \end{array}
   \right.
\end{align*}
with the advantage that $\BBfunction$ is lower-semicontinuous, convex and $1$-homogeneous.

We shall formulate the generalized continuity equation  in terms of measure-valued quantities, namely mass $\mu \in \STMeasp$, momentum $\nu \in \STMeasd$, and source term $\zeta \in \STMeas$.
We will consider curves of measures on $\domain$ instead of just measures on the product space $[0,1] \times \domain$ as the proper
measure theoretic setup for the continuity equation equation with source term. We recall that $D$ is assumed to be a bounded convex domain with Lipschitz boundary.

\begin{Definition}[A weak continuity equation with source term]\label{def:cont-eq}
Let $\mu_A, \mu_B \in \Meas$ be given.
A triple of measures $(\mu, \nu, \zeta)$ in the space $\STMeasp \, \times\,  \STMeasd \, \times \, \STMeas$ is said to be a weak solution to  
 the continuity equation with source term
\begin{align*}
 \partial_t \mu + \div(\nu) = \zeta\, ,
  \qquad \mu_0 = \mu_A\,, \quad \mu_1 =  \mu_B\, ,  
\end{align*} 
if the following conditions hold:
\begin{enumerate}
	\item[\emph{(i)}] 
the measures $\mu, \nu$ and $\sigma$ admit disintegrations with respect to the Lebesgue measure in time: i.e., there exist measure-valued functions
$t \mapsto \mu_t$ weak*-continuous in  $\Meas$\,, 
$t \mapsto \nu_t$ Borel measurable in $\SMeasd $ with
$\int_0^1 |\nu_t|(\domain) \d t < \infty$, and
$t \mapsto \zeta_t$  Borel measurable in $\SMeas$ with $\int_0^1 |\zeta_t|(\domain) \d t < \infty$,
such that
    \begin{align*}
      & \int_{[0,1] \times \domain} \eta(t,x) \d \mu(t,x) = \int_0^1 \int_\domain \eta(t,x) \d \mu_t(x) \d t \quad \forall \eta \in L^1(\mu)\,, \\
      & \int_{[0,1] \times \domain} \eta(t,x) \d \nu(t,x) = \int_0^1 \int_\domain \eta(t,x) \d \nu_t(x) \d t \quad \forall \eta \in L^1(\nu)\,, \\
      & \int_{[0,1] \times \domain} \eta(t,x) \d \zeta(t,x) = \int_0^1 \int_\domain \eta(t,x) \d \zeta_t(x) \d t \quad \forall \eta \in L^1(\zeta)\,.
    \end{align*}
\item[\emph{(ii)}]
 the continuity equation with source term $\partial_t \mu + \div(\nu) = \zeta$ with boundary values $\mu_0 = \mu_A$ and $\mu_1 = \mu_B$ holds in the sense of distributions, 
i.e., for all space-time test functions  $\eta \in C^1_0( [0,1] \times \bar \domain)$,
    \begin{eqnarray}\label{eq:CE-dist}
      0 &=& \int_0^1 \bigg[\int_\domain \partial_t \eta(t,x) \d \mu_t(x) +  \int_\domain \nabla \eta(t,x) \cdot \d \nu_t(x) +  \int_\domain \eta(t,x) \d \zeta_t(x) \bigg] \d t   \\
     &&  -\int_\domain \eta(1,x) \d \mu_B(x) +   \int_\domain  \eta(0,x) \d \mu_A(x)\, .
    \end{eqnarray}
    \end{enumerate}
\end{Definition}
    We denote the set of all solutions of the weak continuity equation with source term by $\CE[0,1]$.

A standard approximation argument (see \cite[Lemma 4.1]{DNS09}) shows that solutions to the continuity equation with source term satisfy, for all $0 \leq t_0 \leq t_1 \leq 1$,
\begin{align}\label{eq:CE-2}
\begin{aligned}
  &   \int_\domain \eta(t_1,x) d \mu_{t_1}(x) -  \int_\domain \eta(t_0,x) d \mu_{t_0}(x) \\
  & = \int_{t_0}^{t_1} \int_\domain \partial_t \eta(t,x) \d \mu_t(x) \d t  + \int_{t_0}^{t_1} \int_\domain \nabla \eta(t,x) \cdot \d \nu_t(x) \d t + \int_{t_0}^{t_1}  \int_\domain \eta(t,x) \d \zeta_t(x) \d t
\end{aligned}
\end{align}
for all space-time test functions  $\eta \in C^1( [0,1] \times \overline{\domain})$.
In particular, taking $\eta(t,x) \equiv 1$, it follows that
\begin{align}\label{eq:mass-increase}
 \mu_{t_1}(\domain) - \mu_{t_0}(\domain) 
  = \int_{t_0}^{t_1} \zeta_t(\domain) \d t\;.
\end{align}

\par\bigskip
Next, we define the energy \eqref{eq:EnergyWithSourceL2L1} in terms of measures. To this end
we decompose for each $t\in [0,1]$ the triple $(\mu_t, \nu_t, \zeta_t) \in \Meas \times \SMeasd \times \SMeas$ using the Lebesgue decomposition theorem into
\begin{align*}
 \mu_t = \u_t \Leb + \mu_t^\perp\;, \qquad
 \nu_t = m_t \Leb + \nu_t^\perp\;, \qquad
 \zeta_t = z_t \Leb + \zeta_t^\perp\; ,
\end{align*}
such that the singular parts $\mu_t^\perp \in \Meas$, $\nu_t^\perp \in \SMeasd$, and $\zeta_t^\perp \in \SMeas$ are singular with respect to the Lebesgue measure $\Leb$ on $\domain$.
Then we define $\Leb_t^\perp := \mu_t^\perp + |\nu_t^\perp| + |\zeta_t^\perp| \in \Meas$, so that $\Leb_t^\perp$ is orthogonal to $\Leb$.
By construction, the singular parts admit a density with respect to $\Leb_t^\perp$:
\begin{align*}
 \mu_t^\perp = \u_t^\perp \Leb_t^\perp\;, \qquad
 \nu_t^\perp = m_t^\perp \Leb_t^\perp \;, \qquad
 \zeta_t^\perp = z_t^\perp \Leb_t^\perp\;.
\end{align*}
With this decomposition at hand, we can define the rigorous version of the energy functional \eqref{eq:EnergyWithSourceL2L1} in the measure-valued setting.
The path energy functional for transport will be taken from the Benamou-Brenier formulation of the $L^2$-Wasserstein distance:
\begin{align*}
 \ABB[\mu_t, \nu_t] :=  \int_\domain \BBfunction(\u_t, m_t) \d \Leb + \int_\domain \BBfunction(\u_t^\perp, m_t^\perp) \d \Leb_t^\perp \, .
\end{align*}
To describe the path energy functional involving  the source term, recall that $\ling : \R \to \R$ is a non-negative, convex function with linear growth (i.e., $\sup_{s \to \infty} \ling(s)/s \in (0,\infty)$ and $\sup_{r \to \infty} \ling(-s)/s \in (0,\infty)$ ), satisfying $\ling(0) = 0$.
We consider the path energy functional for the source term given by  
\begin{align*}
 & \AM[\zeta_t] :=  \left( \int_\domain r(z_t) \d \Leb + \int_\domain |z_t^\perp| \d \Leb_t^\perp \right)^2 \, .
\end{align*}
Note that we consider a $1$-homogeneous integrand for the singular part of the source measure 
with the aim to allow for singular sources with support of the source measure on a set of zero Lebesgue measure.  
The total energy functional $\cA_\penaltyPushforward : \Meas \times \SMeasd\times \SMeas \to [0,+\infty]$ is defined as
\begin{align*}
 \cA_\penaltyPushforward[\mu_t, \nu_t, \zeta_t] := \ABB[\mu_t, \nu_t] + \frac{1}{\penaltyPushforward} \AM[\zeta_t] \; .
\end{align*}

Here we use the proper extension of the functionals $J(u) = \int_D f(u) \d \Leb$ for $u \in L^1(D)$  to the space of Radon measures due to Bouchitt{\'e} and Buttazzo 
\cite{bouchitte1990new, bouchitte1992integral}, 
which is given by 
$$J(u) = \int_D f( \tfrac{\d u}{\d \Leb } ) \d \Leb + \int_D f^{\infty} ( \tfrac{\d u^s}{\d |u^s|} ) \d |u^s| = \int_D f( \tfrac{\d u}{\d \Leb } ) \d \Leb + f^{\infty}(1) |u^s|(D)$$  for $u \in \SMeas$, where $u = \tfrac{\d u}{\d \Leb } + u^s$ is the Lebesgue decomposition of $u$ and $f^{\infty}$ is the recession function of $f$ defined by $f^{\infty}(y) = \sup_{s \to + \infty} f(sy)/s$. In fact, the functional $J$ is then weak*-lower-semicontinuous in $\SMeas$. 

Note that in our case the function $\BBfunction$ is $1$-homogeneous and the recession function of $r$ is up to a constant given by $z \mapsto |z|$, since the function $r$ is assumed to be of linear growth.
Thus, we can rigorously define the total energy functional $\Ed$ (cf. \ref{eq:EnergyWithSourceL2L1}) for measures $(\mu, \nu, \zeta) \in \STMeasp \times \STMeasd \times \STMeas$ 
as
\begin{align} \label{eq:energy}
 \Ed(\mu,\nu,\zeta) 
 = \left\{
    \begin{array}{ll} 
       \int_0^1 \cA_\penaltyPushforward[\mu_t, \nu_t, \zeta_t] \d t & \mbox{if } (\mu, \nu, \zeta) \in \CE[0,1] \, , \\
       + \infty & \mbox{otherwise} \, .
    \end{array}
   \right.
\end{align}

\par\bigskip

As an immediate consequence of general lower-semicontinuity results for integral functionals on measures (\cite{AFP00}, see also \cite[Section 3]{DNS09}) we deduce the lower semi-continuity of $\cA_\penaltyPushforward$.
More precisely, for weak$^*$-convergent sequences of measures 
 \begin{align*}
  \mu_t^n \weaklyStar \mu_t \in \Meas \, , \qquad
  \nu_t^n \weaklyStar \nu_t \in \SMeasd \, , \qquad
  \zeta_t^n \weaklyStar \zeta_t \in \SMeas \, .
 \end{align*}
we have that 
 \begin{align}\label{eq:lsc1}
  \cA_\penaltyPushforward[\mu_t, \nu_t, \zeta_t]  \leq \liminf_{n \to \infty} \cA_\penaltyPushforward[\mu_t^n, \nu_t^n, \zeta_t^n] \, .
 \end{align}
Next, we state a compactness result for solutions to the weak continuity equation with source term.
\begin{Proposition}[Compactness for solutions to the continuity equation with source term with bounded action]\label{prop:compact}
  Suppose that a sequence $(\mu^n, \nu^n, \zeta^n)_{n \in \N}$ in $\CE[0,1]$ with boundary values $\mu_A$ and $\mu_B$ has bounded energy, i.e.,
  \begin{align}\label{eq:BoundedEnergyAssumption}
    \sup_n \Ed[\mu^n, \nu^n, \zeta^n] \d t  \leq C \, .
  \end{align}
Then, there exists a subsequence (again indexed by n) and a triple $(\mu, \nu, \zeta) \in \CE[0,1]$ such that 
$\mu_t^n \weaklyStar \mu_t$ in $\Meas$ for all $t \in [0,1]$\,, \ 
$\nu^n \weaklyStar \nu$ in $\STMeasd$\,, \ $\zeta^n \weaklyStar \zeta$ in $\STMeas$, and
\begin{align}\label{eq:lsc}
  \int_0^1 \cA_\penaltyPushforward[\mu_t, \nu_t, \zeta_t] \d t \leq \liminf_{n \to \infty} \int_0^1 \cA_\penaltyPushforward[\mu_t^n, \nu_t^n, \zeta_t^n] \d t \, .
\end{align}
\end{Proposition}
\begin{proof}
Since $\ling$ is of linear growth, we have $|z| \leq C (1 + \ling(z))$, hence
\begin{align*}
	|\zeta_t^n(D)|
	 = \int_D |z_t^n| \d \Leb  + \int_D |(z_t^n)^\perp| \d \Leb^\perp
	 \leq C \Big(1 + \sqrt{\AM[\zeta_t^n]}\Big)
	 \leq C \;.
\end{align*}
  Here and in what follows $C$ is a generic constant, which may change from line to line.
   Because of the bounded energy assumption \eqref{eq:BoundedEnergyAssumption}, we obtain a uniform bound for the source term:
  \begin{align}\label{eq:M3}
    \sup_n |\zeta^n([0,1] \times \domain)| 
    = \sup_n \int_0^1 |\zeta_t^n(D)| \d t 
     \leq C \, .
  \end{align}
 For this estimate we deduce that a subsequence of $\zeta^n$ converges weakly-$*$ to a measure $\zeta$.
  Crucial for the compactness result is that we can disintegrate $\zeta$ with respect to the Lebesgue measure on $[0,1]$ into a family of measures $\{\zeta_t\}_{t \in [0,1]} \in \SMeas$.
  Now, by \eqref{eq:M3} the sequence $\{t \mapsto |\zeta_t^n|(\domain)\}_n$ is uniformly bounded in $L^2([0,1])$.
  This implies an equi-integrability estimate for $\{t \mapsto \zeta_t^n(\domain)\}_n$, and as a consequence one obtains the requested 
  disintegration $\{\zeta_t\}_{t \in [0,1]} \in \SMeas$ of the limit measure $\zeta$.

  Next, formula \eqref{eq:mass-increase} for the change of mass yields a uniform bound
  \begin{align}\label{eq:mass-bound}
  \mu_t^n(\domain) \leq \mu_A(\domain) + \int_{0}^{t} |\zeta_s^n|(\domain) \d s \leq C
  \end{align}
  for all $n \in \N$ and $t \in [0,1]$.
 
Then, we can proceed exactly as for the source term in $L^2(L^2)$ (cf. \cite{RuSc12}). We include these arguments for the sake of completeness.
First, we claim that the maps $\{t \mapsto |\bfnu_t^n|(\domain)\}_n$ are uniformly bounded in $L^2([0,1])$, hence uniformly integrable.
To see this, we follow \cite[Proposition 3.6]{DNS09} to obtain
\begin{align*}
 &|\bfnu_t^n|(D)
  =  \int_D |m_t^n| \dd \Leb  
        +  \int_D |(m_t^n)^\perp| \dd \Leb^\perp 
 \\& \leq \bigg( \int_D \phi(\u_t^n, m_t^n) \dd \Leb  \bigg)^{\frac12}
 \bigg( \int_D \u_t^n \dd \Leb \bigg)^{\frac12}
 + 
 \bigg( \int_D \phi\big((\u_t^n)^\perp, (m_t^n)^\perp\big)  \dd \Leb^\perp   \bigg)^{\frac12}
 \bigg( \int_D  (\u_t^n)^\perp \dd \Leb^\perp  \bigg)^{\frac12}
\\& \leq \bigg( \int_D \phi(\u_t^n, m_t^n) \dd \Leb   
    +  \int_D \phi\big((\u_t^n)^\perp, (m_t^n)^\perp\big) \dd \Leb^\perp  \bigg)^{\frac12}
 \bigg( \int_D \u_t^n \dd \Leb   + 
  \int_D (\u_t^n)^\perp \dd \Leb^\perp  \bigg)^{\frac12}
\\& = \Big(\ABB(\mu_t^n, \bfnu_t^n)\Big)^{\frac12} \big(\mu_t^n(D)\big)^{\frac12}\,,
\end{align*}
where we used the scalar inequality $\sqrt{ab} + \sqrt{cd} \leq \sqrt{a+c}\sqrt{b+d}$ which holds for $a,b,c,d \geq 0$.
In view of \eqref{eq:BoundedEnergyAssumption} and \eqref{eq:mass-bound}, the claim follows.
Using the inequality $|\bfnu^n|([0,1] \times \domain) \leq (\int_0^1 |\bfnu_t^n|(\domain)^2 \dd t)^{\frac12}$, we infer that the measures $\{\bfnu^n\}_n \in \STMeasd$ have uniformly bounded total variation on $[0,1] \times \domain$. Therefore, we can extract a subsequence that converges weakly$^*$ to some measure $\bfnu \in \STMeasd$. 
Since the mapping $\{t \mapsto |\bfnu_t^n|(\domain)\}_n$ is uniformly integrable,  the image measure of $\bfnu$ under the mapping $(t,x) \mapsto t$ is absolutely continuous with respect to the Lebesgue measure on $[0,1]$. Therefore, we obtain a disintegration $\{\bfnu_t\}_{t \in [0,1]} \in \SMeasd$ of $\nu$.

 Fix $0 \leq \tau \leq 1$, take $\eta \in C^1(\domain)$, and set $ \xi(t,x) := \nabla\eta(x)\chi_{[0,\tau]}(t)$. 
  Even though $\bar\xi$ is discontinuous, it follows from general approximation results (see \cite[Proposition 5.1.10]{ambrosio2006gradient}) that 
  \begin{align}\label{eq:nu-conv}
  \int_{0}^{\tau} \int_\domain \nabla \eta \dd \bfnu_t^n \dd t
    =   \int_{[0,1] \times \domain}  \xi \d \bfnu^n
    \to \int_{[0,1] \times \domain}  \xi \d \bfnu
    = \int_{0}^{\tau} \int_\domain \nabla \eta \d \bfnu_t \dd t\;.
  \end{align}
Setting $\iota(t,x) := \eta(x) \chi_{[0,\tau]}(t)$ and arguing as above, we obtain
  \begin{align}\label{eq:zeta-conv}
    \int_{0}^{\tau} \int_\domain \eta \dd \zeta_t^n \dd t
    =   \int_{[0,1] \times \domain} \iota\dd \zeta^n
    \to \int_{[0,1] \times \domain} \iota \dd \zeta
    = \int_{0}^{\tau} \int_\domain \eta \dd \zeta_t \dd t\;.
  \end{align}
Now we can obtain convergence of a subsequence of $\{\mu_t^n\}_n$. Indeed,  in view of \eqref{eq:nu-conv} and \eqref{eq:zeta-conv} 
we can pass to the limit in  the weak continuity equation  \eqref{eq:CE-2} for $(\mu^n, \nu^n, \zeta^n)$ and obtain 
weak$^*$-convergence of $\{\mu_\tau^n\}_n$ for a subsequence to some measure $\mu_\tau$ for every $\tau \in [0,1]$.
Finally, we define $\mu \in \STMeasp$ by 
  \begin{align}
   \int_{[0,1] \times \domain} \eta(t,x) \d \mu(t,x) = \int_0^1 \int_\domain \eta(t,x) \d \mu_t(x) \d t \quad \forall \eta \in C_c^0(\domain) \, .
  \end{align}
  It is straightforward to check that $\mu^n$ converges weakly$^*$ to $\mu$ in $\STMeasp$.
 The lower semicontinuity estimate \eqref{eq:lsc} follows from \eqref{eq:lsc1}.
\end{proof}
\begin{Remark}\label{remark:L1L1}
 At first glance the penalty functional  $\int_0^1 \left( \int_\domain |z| d \Leb + \int_\domain |z^\perp| \d \Leb^\perp  \right) \d t$ seems to be an appropriate choice, which 
 allows for singular sources due to the built-in $1$-homogeneity of the integrand.
 However, there is no equi-integrability estimate for a sequence of source terms $\{t \mapsto \zeta_t^n(\domain)\}_n$. In fact, a uniform bound in $L^1$ does not suffice to deduce uniform integrability.
 Thus, the disintegration of the limit measure $\zeta$ remains unclear.
 In other words, there exists a subsequence of an energy minimizing sequence that converges weakly$^*$ to a measure on $[0,1] \times D$, but the limit measure can not necessarily be represented in terms of a curve in $\SMeas$.
\end{Remark}

Now, we can rigorously define a generalized Wasserstein distance.
For $\mu_A, \mu_B \in \Meas$ we define $\Wd[\mu_A,\mu_B] \in [0,+\infty]$ by
 \begin{align}\label{eq:def-Wd}
  \Wd[\mu_A, \mu_B] := \inf_{(\mu,\nu,\zeta) \in \CE[0,1], \mu_0 = \mu_A, \mu_1 = \mu_B} \bigg( \Ed[\mu,\nu,\zeta] \bigg)^{1/2}\, .
 \end{align}
The following result shows in particular that $\Wd[\mu_A,\mu_B] \in [0,\infty)$ for all $\mu_A, \mu_B \in \Meas$.

\begin{Theorem}[Existence of geodesics]\label{thm:geod}
 Let $\penaltyPushforward \in (0,\infty)$ and take $\mu_A, \mu_B \in \Meas$.
 Then there exists a minimizer $(\overline\mu_t, \overline\nu_t, \overline\zeta_t)_{t \in[0,1]}$ that realizes the infimum in \eqref{eq:def-Wd}. 
 Moreover, $\Wd$ defines a metric on $\Meas$, and the associated curve $(\overline\mu_t)_{t \in[0,1]}$ is a constant speed geodesic for $\Wd$, \ie, 
 \begin{align*}
  \Wd[\overline\mu_s, \overline\mu_t] = |s-t| \Wd[\mu_A, \mu_B] 
 \end{align*}
 for all $s,t \in [0,1]$. Furthermore, we have the alternative characterization
 \begin{align*}
  \Wd[\mu_A, \mu_B]
    := \inf_{\mu,\nu,\zeta} \bigg\{ \int_0^1  \sqrt{\cA_\penaltyPushforward[\mu_t, \nu_t, \zeta_t] }  \d t \, : \,
        \begin{array}{c}
         (\mu_t, \nu_t, \zeta_t)_{t \in[0,1]} \in \CE[0,1] \, ,\\
         \mu_0 = \mu_A \, , \; \mu_1 = \mu_B 
        \end{array}
   \bigg\} \, .
 \end{align*}
\end{Theorem}
\begin{proof}

 The linear interpolation $\left\{ \mu_t = (1-t) \mu_A + t \mu_B \right\}_{t \in [0,1]}$ together with $\nu = 0$ and $\zeta = \mu_B - \mu_A$ is an admissible triple for the set $\CE[0,1]$ with finite energy, since the assumptions on $\ling$ imply that
 \begin{align}
  \Ed[\mu,\nu,\zeta] \leq C \big(1 + |\mu_B - \mu_A|(\domain)\big) ^2 < \infty \, .
 \end{align}
It follows that $\Wd[\mu_A, \mu_B] < \infty$, and the existence of a minimizer is an immediate consequence of Proposition \ref{prop:compact}. 
 The remaining statements follow by standard arguments, see \cite[Theorem 5.4]{DNS09} for details.
\end{proof}

\section{Proximal splitting algorithm}\label{sec:numeric}
In this section we derive a numerical scheme to compute geodesics for our new distance introduced in \eqref{eq:def-Wd} for $d=2$.
To this end, we will adapt the proximal splitting algorithm, which was proposed by Papadakis et al. \cite{papadakis2014optimal} for the classical $L^2$ optimal transport problem.
In detail,  the constraint optimization problem is first rewritten as a non-constraint minimization problem adding the indicator function of the set of solutions of the continuity equation 
$\CE[0,1]$ to the cost functional.
Then, the proximal splitting algorithm yields a solution scheme, which only requires to solve a space-time elliptic problem and to project pointwise onto a convex set.
The resulting algorithm is equivalent to the augmented Lagrangian approach in \cite{BeBr00}.
Different from  \cite{BeBr00, papadakis2014optimal} we will use a finite element discretization instead of finite differences.

Let us briefly recall the definition and the basic properties of a proximal mapping (see for instance \cite{combettes2011proximal, papadakis2014optimal}).
In the following, let $(X. \| \cdot \|_X)$ be a Hilbert space and $f:X \to \R \cup \{ \infty \}$ a convex and lsc function.
Then the proximal mapping of $f$ is defined as
 \begin{align}
  \prox_f(x) = \argmin_{y \in X} f(y) + \frac{1}{2} \|x-y\|^2 \, .
 \end{align}
In the sequel it will be important to compute the proximal mapping of the indicator function $\indicatorFct_K$ of a convex set $K \subset X$, which is just given by
$ \prox_{\indicatorFct_K}(x) =  \projection_K(x)\,$,
 where $\projection_K$ is the orthogonal projection on $K$ with respect to the norm $\| \cdot \|_X$. 
Now, we suppose that $D$ is a polygonal domain and 
consider a tetrahedral mesh $S_h$ with grid size $h$ for the space time domain $[0,1]\times D$, which is generated from a triangular mesh for the domain $D$ via subdivision of 
prisms $(k\,h, (k+1)h) \times T$ (with $T$ being a triangle) into $3$ tetrahedrons such that the resulting tetrahedral mesh is an admissible triangulation in space time.
On this triangulation we define fintie element spaces
\begin{align}
 & V^1( S_h ) = \{ \phi_h : [0,1] \times \domain \to \R \; : \; \phi_h \text{ continuous and piecewise linear on elements in } S_h \} \, ,\\
 & V^0( S_h ) = \{ \u_h : [0,1] \times \domain \to \R \; : \; \u_h \text{ piecewise constant on elements in } S_h \} \,.
\end{align}
This allows us define discretization 
\begin{align*}
 \u_h \in V^0(S_h) \, , \quad
 m_h \in \left( V^0(S_h) \right)^d \, , \quad
 z_h \in V^1(S_h) \, ,
\end{align*} for the measures for mass, momentum and source, respectively. 
Furthermore, we will use the notation $p_h = (\u_h, m_h) \in V^0(S_h)^{d+1}$.
For a triple $(\u_h, m_h, z_h) \in V_h^0(S_h) \times V_h^0(S_h)^d \times V_h^1(S_h)$ we choose a weighted $L^2$ norm
\begin{align*}
  \left\| (\u_h, m_h, z_h) \right\| := 
  \left( \int_0^1 \int_\domain |\u_h|^2 + |m_h|^2 + \frac{1}{\penaltyPushforward} |z_h|^2 \d x \d t\right)^{\frac{1}{2}} \, ,
\end{align*}
which can be computed exactly by choosing a quadrature rule of at least second order.
In correspondence to Definition \ref{def:cont-eq} the set of discrete solutions of a continuity equation is defined as follows:
\begin{Definition}
Let $\u_A, \u_B \in V^0(S_h)$ be given. 
Then, the set $\CE_h$ of solutions of a weak continuity equation with source term and boundary values $\u_A, \u_B$ is given by all triples $(\u_h,m_h,z_h) \in V_h^0(S) \times V_h^0(S)^d \times V_h^1(S)$ satisfying
\begin{align*}
 \int_0^1 \int_\domain \u_h \partial_t \phi_h + m_h \nabla_x \phi_h + z \phi_h \d x \d t = \int_\domain \phi_h(1) \u_B - \phi_h(0) \u_A \d x 
 \quad \forall \phi_h \in V^1( S_h ) \, .
\end{align*}
\end{Definition}
Note that we used Neumann boundary condition in space.
The approach can easily be adopted in case of Dirichlet or periodic boundary conditions.

Now, we can state a discrete version of the minimization problem \eqref{eq:def-Wd}:
\begin{align}
 \inf\limits_{(\u_h, m_h,z_h) \in \CE_h } 
 \left( \int_0^1 \int_\domain \Phi  (\u_h, m_h) \d x \d t + \frac{1}{\penaltyPushforward} \int_0^1 \left( \int_\domain \Ling(z_h) \d x \right)^2 \d t \right) \, ,
\end{align}
where $\Ling[z_h]$ denotes a suitable interpolation of $\ling(z_h)$. Here, we define 
$\Ling[z_h](t,x)$ as the piecewise affine interpolation of  $\ling(z_h((k-1)h,\cdot))$ on the triangle $T$ 
for $(t,x) \in (kh,(k+1)h) \times T$ (one of the prisms underlying the tetrahedral grid).
Numerically, we are not able to treat singular measures as presented in Section \ref{sec:existence}.
Our concrete choices of $r(s)$, which coincide with $|s|$ for large $s$ allow to approximate such measures in the source term cost 
supported on the union of the support of basis functions. Thus point or line sources are numerically treated via sources with support thickness $2h$.
To apply a proximal splitting algorithm, we split the functional into
\begin{align}
 & F_1(\u_h,m_h,z_h) := F_{\text{trans}}(\u_h,m_h) + \frac{1}{\penaltyPushforward} F_{\text{source}}(z_h) \\
  &\qquad  \mbox{with } F_{\text{trans}}(\u_h,m_h) := \int_0^1\!\! \int_\domain \Phi (\u_h, m_h) \d x \d t \,,\;
   F_{\text{source}}(z_h) := \int_0^1\!\! \left( \int_\domain \Ling[z_h] \d x \right)^2 \!\!\d t\,, \\
  & F_2(\u_h,m_h,z_h) := I_{\CE_h}(\u_h,m_h,z_h) \, .
\end{align}
Next, let us compute the proximal mappings of $F_1$ and $F_2$.\medskip

\noindent \textbf{Proximal map of $F_2$.}
The computation of the proximal mapping of the indicator function of $\CE_h$ requires the orthogonal projection of a point 
$\left( p_h = (\u_h, m_h), z_h \right ) \in V^0(S)^{d+1} \times V^1(S)$ onto $\CE_h$, i.e. we ask for  
$(p^\ast_h, z^\ast_h) \in  \argmin_{(q_h, w_h) \in \CE_h} \left\| (p_h, z_h) - (q_h, w_h) \right\|^2$.
The associated  Lagrangian is given by
\begin{align}
 L[(q_h,w_h),\psi_h] 
 = \left\| (p_h, z_h) \!-\! (q_h, w_h) \right\|^2 \!
 - \! \int\limits_0^1\!\! \int\limits_\domain\! q_h \cdot\nabla_{\!(t,x)} \psi_h\! +\! w_h \psi_h \d x \d t +\! \int\limits_\domain \!\psi_h(1) \u_B\! -\! \psi_h(0) \u_A \d x \, ,
\end{align}
with a Lagrange multiplier $\psi_h \in V^1(S_h)$. In terms of this Lagrangian 
the projection problem can be written as a saddle point problem, where as ask for $(p^\ast_h, z^\ast_h, \phi^\ast_h) \in V^0(S)^{d+1} \times V^1(S) \times V^1(S)$, such that
\begin{align}
L[(p^\ast_h, z^\ast_h, \phi^\ast_h)] = \min_{ (q_h, w_h) \in V_h^0(S)^{d+1} \times V_h^1(S) } \max_{\psi_h \in V_h^1(S) } L[(q_h,w_h,\psi_h)] \, .
\end{align}
The Euler-Lagrange equations corresponding to this saddle point problem are given by
\begin{alignat}{2}
 & \int_0^1 \int_\domain p^\ast_h \cdot \nabla_{\!(t,x)} \psi_h + z^\ast_h \psi_h \d x \d t = \int_\domain \psi_h(1) \,\u_B - \psi_h(0)\, \u_A \d x \quad && \forall \psi_h \in V_h^1(S) \label{EL-1}\\
 & \int_0^1 \int_\domain  q_h \cdot \nabla_{\!(t,x)} \phi^\ast_h\d x \d t = \int_0^1 \int_\domain 2 (p^\ast_h - p_h ) \, q_h \d x \d t \quad && \forall q_h \in V_h^0(S)^{d+1} \label{EL-2} \\
 & \int_0^1 \int_\domain \phi^\ast_h w_h \d x \d t = \int_0^1 \int_\domain \frac{2}{\penaltyPushforward} (z^\ast_h - z_h) \,w_h \d x \d t \quad && \forall w_h \in V_h^1(S) \label{EL-3}
\end{alignat}
Testing equation \eqref{EL-2} with $q_h = \nabla_{\!(t,x)}  \psi_h$ and then using equation \eqref{EL-1} gives
\begin{align}
 \int_0^1 \int_\domain \frac{1}{2} \nabla_{\!(t,x)}  \phi^\ast_h \cdot \nabla_{\!(t,x)}  \psi_h \d x \d t
 = & \int_0^1 \int_\domain (p^\ast_h - p_h) \cdot \nabla_{\!(t,x)}  \psi_h \d x \d t \\
 = & \int_\domain \psi_h(1) \u_B - \psi_h(0) \u_A \d x - \int_0^1 \int_\domain z^\ast_h \psi_h + p_h \cdot \nabla_{\!(t,x)}  \psi_h \d x \d t
\end{align}
Hence, by using equation \eqref{EL-3}  ( $z^\ast_h = z_h + \frac{\penaltyPushforward}{2} \phi^\ast_h$) we obtain
\begin{align}
 \int_0^1 \int_\domain \frac{1}{2} \nabla_{\!(t,x)}  \phi^\ast_h \nabla_{\!(t,x)}  \psi_h + \frac{\penaltyPushforward}{2} \phi^\ast_h \psi_h \d x \d t = &
 \int_\domain \psi_h(1) \u_B - \psi_h(0) \u_A \d x \\
 & - \int_0^1 \int_\domain z_h \psi_h + p_h \nabla_{\!(t,x)}  \psi_h \d x \d t
\end{align}
for all $\psi_h \in V_h^1(S) $.\\
After computing $\phi_h$ the solution of the projection problem is given by
\begin{align}
p^\ast_h  = p_h + \frac{1}{2} \nabla_{\!(t,x)}  \phi^\ast_h \,,\quad
z^\ast_h  = z_h + \frac{\penaltyPushforward}{2} \phi^\ast_h\,.
\end{align}

\noindent \textbf{Proximal map of $F_1$.} 
The transport term $F_{\text{trans}}$ does only depend on $\u_h$ and $m_h$ and can be treated exactly as for classical optimal transport.
Since we observe pointwise that $\Phi^{*} = I_K$ is an indicator function of the convex set  
$
 K = \left\{ (\u,m) \; : \; \u + \frac{|m|^2}{4} \leq 0 \right\} \, ,
$
(see \cite{BeBr00}) we can use Moreau's identity and get
\begin{align}
 \prox_{\Phi} (\u, m) = (\u,m) - \prox_{\Phi^{*}} (\u, m ) = (\u,m) - \projection_K \left( \u, m \right) \; .
\end{align}
The projection onto $K$ separately on each tetrahedron of the simplicial mesh $S_h$ 
due to the choice of our finite element spaces with $p_h \in V^0(S_h)^{d+1}$.

We note that for a source term in $L^2$ both in space and time we easily get a pointwise update
\begin{align}
 \prox_{ \frac{\gamma}{\penaltyPushforward} |\cdot|^2 } (z) 
 = \argmin_{w} \frac{1}{\penaltyPushforward} |w|^2 + \frac{1}{\penaltyPushforward} |w-z|^2
 = \frac{1}{1 + \gamma} z \, .
\end{align}
Following the computation in \cite{esser2009applications} we also get a pointwise update for the proximal operator of the source term in $L^1(L^1)$, which is given by
\begin{align}
  \prox_{ \frac{\gamma}{\penaltyPushforward} |z| } (z) 
  = \left\{ 
    \begin{array}{cl}
     0 \, , & \mbox{if } |z| \leq \frac{\gamma}{2} \, \\
      z - \frac{\gamma}{2} sgn(z) \, , & \mbox{else. }
    \end{array}
    \right.
\end{align}

Thus, a numerical scheme for a source term in $L^1(L^1)$ would be as simple as for a source term in $L^2(L^2)$, but existence of geodesics is not guaranteed (see Remark~\ref{remark:L1L1}).
In case of a linear growth function $\ling(\cdot)$ the minimization problem only decouples in time but not in space.
Hence, for each discrete time step $k$  we have to solve 
\begin{align}
 \argmin_{w_h } \frac{\gamma}{\penaltyPushforward} \left( \int_\domain \Ling[w_h](kh,x)\d x \right)^2
 + \frac{1}{2 \penaltyPushforward} \int_\domain |w_h(kh,x) - z_h(kh,x)|^2 \d x \, .
\end{align}
For a source term in $L^2(L^1)$ the minimization problem is well defined, but since $r(z) = |z|$ is not differentiable it is not clear how to find the minimizer.
Therefore we restrict our numerical computations to the case of $r$ being the Huber function and use a gradient descent to compute this minimum. 
\medskip

\noindent \textbf{Douglas-Rachford splitting algorithm.} 
Finally, to solve the minimization problem
\begin{align}
(p^\ast_h, z^\ast_h) \in \argmin_{(q_h,w_h) \in V^0(S)^{d+1} \times V^1(S)} F_1(q_h,w_h) + F_2(q_h,w_h) 
\end{align}
we apply the Douglas-Rachford splitting algorithm \cite{eckstein1992douglas}, which is given by the iteration 
\begin{align}\label{algo:douglasrachford}
 (q_h^{n},w_h^{n}) & = \prox_{\gamma F_2} ( (p_h^{n-1},z_h^{n-1}) ) \, ,\\
 (p_h^n, z_h^n)  & = (p_h, z_h)^{n-1} + \alpha \left( \prox_{\gamma F_1}(2 (q_h^n, w_h^n) - (p_h, z_h)^{n-1}) - (q_h, w_h)^n \right) \, ,
\end{align} 
for an initial value $(p^0_h, z^0_h)$ and a step size weight $\alpha$ have to be chosen.
It is guaranteed that for $\gamma > 0$ and $\alpha \in (0,2)$ the sequence $(p^n_h, z^n_h)$ as well as $(q^n_h, w_h)^n$ converges to a solution of the minimization problem.

\section{Numerical results}\label{sec:numericalResults}
We have applied the proposed scheme for the optimal transport with source term for different sets of $(\u_A,\u_B)$.
In all computations $D=(0,1)^2$ and the grid size is $h=2^{-7}$.
At first, we demonstrate that the density function $r$ for the source term is the right choice to deal with approximations of singular measures as source terms.
To this end, we consider in Figure~\ref{fig:singularMeasure} measures $\u_A$ and $\u_B$ supported on a thin rectangular strip with constant but different density.
The proposed model with the $L^2$--Huber ($L^2(H)$) type cost functional $\int_0^1 (\int _D r(\u) \d x )^2 \d t$ for the source term is 
able to generate the required singular measure and 
the geodesic Wasserstein geodesic is just given by a blending of the two measure $\u_A$ and $\u_B$. 
The generating of singular sources is not possible for an $L^2(L^2)$ type cost functional in space time,
which was proposed in \cite{RuSc12}.
Indeed, chosen the cost functional $\int_0^1 \int_D z^2 \d x \d t$ for the same data, the generation of mass via the source term takes place on a thick super set of the 
rectangular strip and then transported toward to strip.
We also observe a similar effect for absolutely continuous measures. In Figure~\ref{fig:squareMeasure} we compare the $L^2(H)$ and the $L^2(L^2)$ source term for geodesics 
connecting differently scaled characteristic functions of a square. 
Figure~\ref{fig:squareMeasureSource} shows a plot of $t \mapsto \int_D |z(t, \cdot )| \d x$ for both models.
Let us remark, that numerical diffusion in particular on coarse meshes leads to a blurring effect for the source term at discontinuities of the density which 
is then accompanied by minor transport to compensate for this diffusion.

\begin{figure}[!h]
\setlength{\unitlength}{0.05\textwidth}
\centering
\resizebox{1.0\linewidth}{!}{
  \begin{minipage}[h]{1.1\textwidth}
  \begin{picture}(20,2.4)
    \put(0,1){ $L^2(H)$ }
    \put(2,0){\fbox{\includegraphics[width=0.08\textwidth]{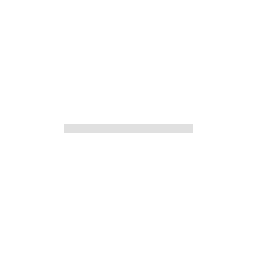}}}
    \put(4.2,0){\fbox{\includegraphics[width=0.08\textwidth]{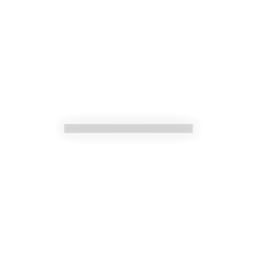}}}
    \put(6.4,0){\fbox{\includegraphics[width=0.08\textwidth]{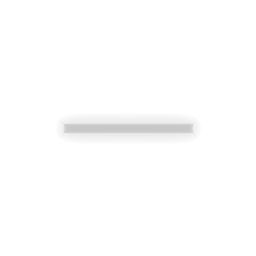}}}
    \put(8.6,0){\fbox{\includegraphics[width=0.08\textwidth]{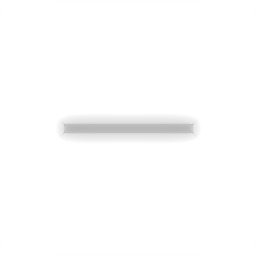}}}
    \put(10.8, 0){\fbox{\includegraphics[width=0.08\textwidth]{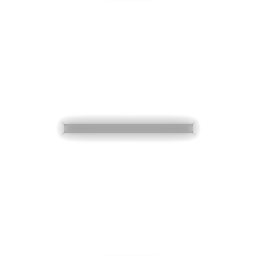}}}
    \put(13,0){\fbox{\includegraphics[width=0.08\textwidth]{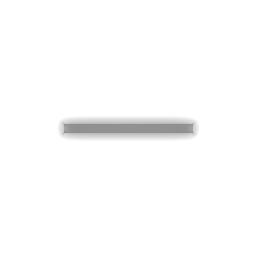}}}
    \put(15.2,0){\fbox{\includegraphics[width=0.08\textwidth]{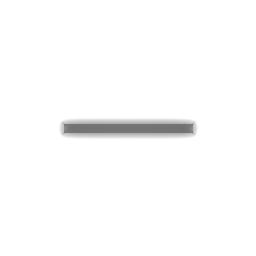}}}
    \put(17.4,0){\fbox{\includegraphics[width=0.08\textwidth]{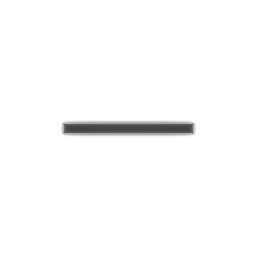}}}
    \put(19.6,0){\fbox{\includegraphics[width=0.08\textwidth]{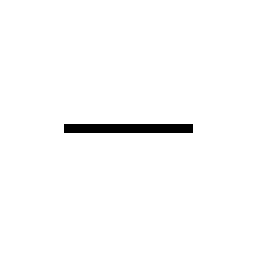}}}
  \end{picture}
  \end{minipage}
}
\resizebox{1.0\linewidth}{!}{
  \begin{minipage}[h]{1.1\textwidth}
  \begin{picture}(20,2.4)
    \put(0,1){ $L^2(L^2)$ }
    \put(2,0){\fbox{\includegraphics[width=0.08\textwidth]{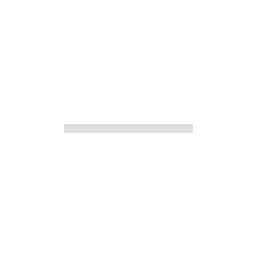}}}
    \put(4.2,0){\fbox{\includegraphics[width=0.08\textwidth]{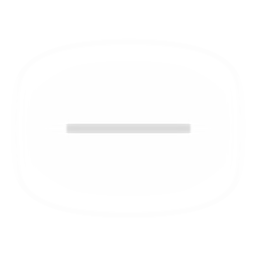}}}
    \put(6.4,0){\fbox{\includegraphics[width=0.08\textwidth]{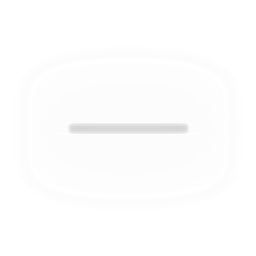}}}
    \put(8.6,0){\fbox{\includegraphics[width=0.08\textwidth]{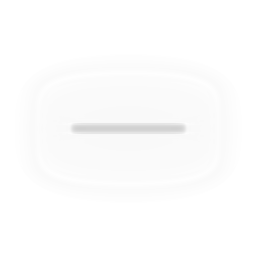}}}
    \put(10.8, 0){\fbox{\includegraphics[width=0.08\textwidth]{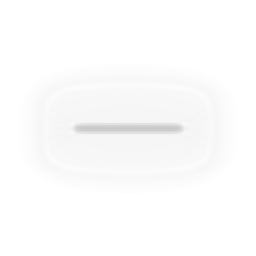}}}
    \put(13,0){\fbox{\includegraphics[width=0.08\textwidth]{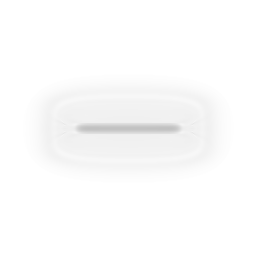}}}
    \put(15.2,0){\fbox{\includegraphics[width=0.08\textwidth]{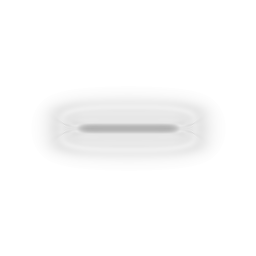}}}
    \put(17.4,0){\fbox{\includegraphics[width=0.08\textwidth]{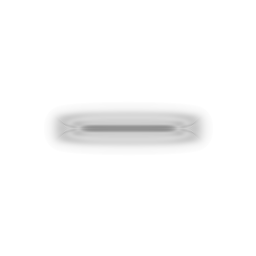}}}
    \put(19.6,0){\fbox{\includegraphics[width=0.08\textwidth]{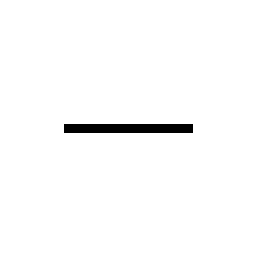}}}
  \end{picture}
  \end{minipage}
}
\caption{Optimal transport geodesic between approximations of singular measures with different intensity. Here the source term parameter is $\penaltyPushforward = 10^{0}$.}
\label{fig:singularMeasure}
\end{figure}

\begin{figure}[!h]
\setlength{\unitlength}{0.05\textwidth}
\centering
\resizebox{1.0\linewidth}{!}{
  \begin{minipage}[h]{1.1\textwidth}
  \begin{picture}(20,4.4)
    \put(0,3){ $L^2(H)$ }
    \put(2,2.4){\fbox{\includegraphics[width=0.08\textwidth]{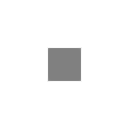}}}
    \put(4.2,2.4){\fbox{\includegraphics[width=0.08\textwidth]{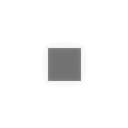}}}
    \put(6.4,2.4){\fbox{\includegraphics[width=0.08\textwidth]{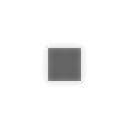}}}
    \put(8.6,2.4){\fbox{\includegraphics[width=0.08\textwidth]{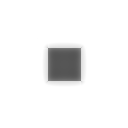}}}
    \put(10.8,2.4){\fbox{\includegraphics[width=0.08\textwidth]{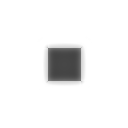}}}
    \put(13,2.4){\fbox{\includegraphics[width=0.08\textwidth]{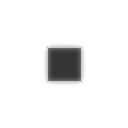}}}
    \put(15.2,2.4){\fbox{\includegraphics[width=0.08\textwidth]{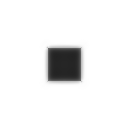}}}
    \put(17.4,2.4){\fbox{\includegraphics[width=0.08\textwidth]{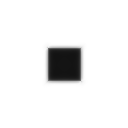}}}
    \put(19.6,2.4){\fbox{\includegraphics[width=0.08\textwidth]{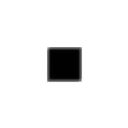}}}
    \put(4.2,0){\fcolorbox{white}{black}{\includegraphics[width=0.08\textwidth]{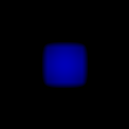}}}
    \put(6.4,0){\fcolorbox{white}{black}{\includegraphics[width=0.08\textwidth]{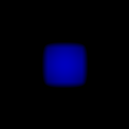}}}
    \put(8.6,0){\fcolorbox{white}{black}{\includegraphics[width=0.08\textwidth]{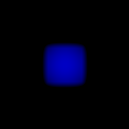}}}
    \put(10.8, 0){\fcolorbox{white}{black}{\includegraphics[width=0.08\textwidth]{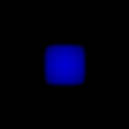}}}
    \put(13,0){\fcolorbox{white}{black}{\includegraphics[width=0.08\textwidth]{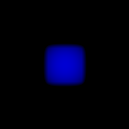}}}
    \put(15.2,0){\fcolorbox{white}{black}{\includegraphics[width=0.08\textwidth]{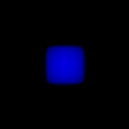}}}
    \put(17.4,0){\fcolorbox{white}{black}{\includegraphics[width=0.08\textwidth]{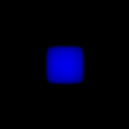}}}
    \put(19.6,0){\fcolorbox{white}{black}{\includegraphics[width=0.08\textwidth]{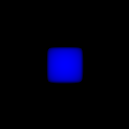}}}
  \end{picture}
  \end{minipage}
}
\resizebox{1.0\linewidth}{!}{
  \begin{minipage}[h]{1.1\textwidth}
  \begin{picture}(20,4.8)
    \put(0,3){ $L^2(L^2)$ }
    \put(2,2.4){\fbox{\includegraphics[width=0.08\textwidth]{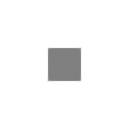}}}
    \put(4.2,2.4){\fbox{\includegraphics[width=0.08\textwidth]{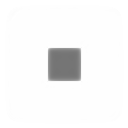}}}
    \put(6.4,2.4){\fbox{\includegraphics[width=0.08\textwidth]{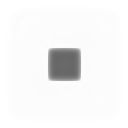}}}
    \put(8.6,2.4){\fbox{\includegraphics[width=0.08\textwidth]{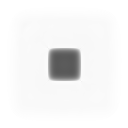}}}
    \put(10.8,2.4){\fbox{\includegraphics[width=0.08\textwidth]{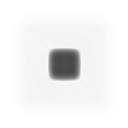}}}
    \put(13,2.4){\fbox{\includegraphics[width=0.08\textwidth]{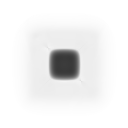}}}
    \put(15.2,2.4){\fbox{\includegraphics[width=0.08\textwidth]{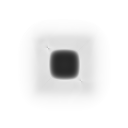}}}
    \put(17.4,2.4){\fbox{\includegraphics[width=0.08\textwidth]{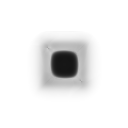}}}
    \put(19.6,2.4){\fbox{\includegraphics[width=0.08\textwidth]{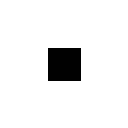}}}
    \put(4.2,0){\fcolorbox{white}{black}{\includegraphics[width=0.08\textwidth]{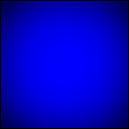}}}
    \put(6.4,0){\fcolorbox{white}{black}{\includegraphics[width=0.08\textwidth]{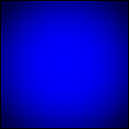}}}
    \put(8.6,0){\fcolorbox{white}{black}{\includegraphics[width=0.08\textwidth]{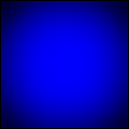}}}
    \put(10.8, 0){\fcolorbox{white}{black}{\includegraphics[width=0.08\textwidth]{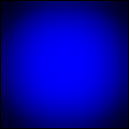}}}
    \put(13,0){\fcolorbox{white}{black}{\includegraphics[width=0.08\textwidth]{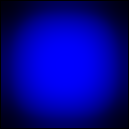}}}
    \put(15.2,0){\fcolorbox{white}{black}{\includegraphics[width=0.08\textwidth]{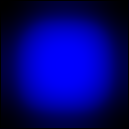}}}
    \put(17.4,0){\fcolorbox{white}{black}{\includegraphics[width=0.08\textwidth]{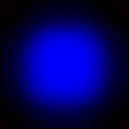}}}
    \put(19.6,0){\fcolorbox{white}{black}{\includegraphics[width=0.08\textwidth]{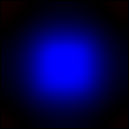}}}
  \end{picture}
  \end{minipage}
}
\caption{Optimal transport geodesic and corresponding source terms between two characteristic functions of squares with different intensity. Here the source term parameter is $\penaltyPushforward = 10^{0}$.}
\label{fig:squareMeasure}
\end{figure}

 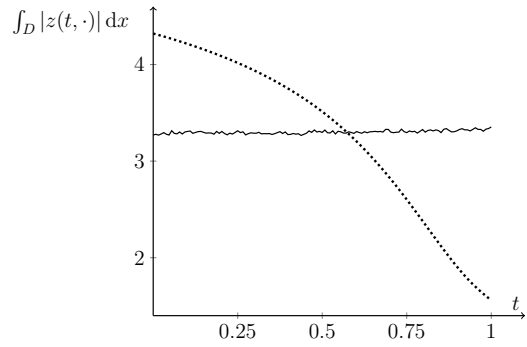
\begin{figure}[h!]
\hspace*{16ex}
\resizebox{0.4\linewidth}{!}{
\begin{minipage}{0.7\linewidth}
\begin{tikzpicture}
   \begin{axis}
     [ width=1.0\linewidth,
      at={(0.0\linewidth,0)},
        xmin=0.0, xmax=1.1,
        ymin=1.4, ymax=4.6,
        every axis/.append style={font=\Large},
        axis lines=middle,
        axis line style = thick,
        axis line style={->},
        label style = {font=\normalsize},
        y label style={anchor=east},
        y label style={at={(-0.05,0.95)}},
        xlabel={\Large  $t$},
        ylabel={ \Large $\int_D |z(t, \cdot )| \d x$},
        xtick={0.25,0.5,0.75,1},
        ytick={2,3,4},
     ]
     \addplot [thick] table[x expr=\thisrowno{0} / 128, y expr=\thisrowno{1} * 100 / 16384] {images/generateMass/Sqaure-Width1_4/Huber_eps1-4_del1/L1NormSourceTerm.txt};
     \addplot [dotted, ultra thick] table[x expr=\thisrowno{0} / 128, y expr=\thisrowno{1} * 100 / 16384] {images/generateMass/Sqaure-Width1_4/L2_del1/L1NormSourceTerm.txt};
    \end{axis}
\end{tikzpicture}
\end{minipage}
}
\caption{Distribution of the $L^1$-norm of the source term in time for the example in Figure~\ref{fig:squareMeasure}. (continuous line: $L^2(H)$, dotted line: $L^2(L^2)$) }
\label{fig:squareMeasureSource}
\end{figure}

Next, we investigate the effect of the source term parameter $\penaltyPushforward$ for the $L^2(H)$ model.
In Figure~\ref{fig:effectDeltaSquare} we choose as input data $\u_A$ at time $t=0$ a characteristic function of a square and as input data $\u_B$ at time $t=1$ the same measure density with an additional characteristic function of a translated square of the same size.  Now, optimizing the connecting path with respect to the generalized Wasserstein distance there is a competition between the curve which simply blends the second square and a curve which transports part of the second square and blends of remaining non transported measure. 
This balance between both processes depends on  $\penaltyPushforward$. 
In the limit $\penaltyPushforward \to 0$ transport becomes cheaper, 
which is reflected by the computational results for small $\delta$. 
In contrast for $\penaltyPushforward \to \infty$ transport becomes expensive 
and a simple blending can be observed for large values of $\penaltyPushforward$ in Figure~\ref{fig:effectDeltaSquare}. 
A similar effect is shown in Figure~\ref{fig:SplittingBumpFunctions}, where the a bump function in the center of the images is transported to a splitted bump function in the corners.

\begin{figure}[!h]
\setlength{\unitlength}{0.05\textwidth}
\centering
\resizebox{1.0\linewidth}{!}{
  \begin{minipage}[h]{1.0\textwidth}
  \begin{picture}(20,2.4)
    \put(0,0){\includegraphics[width=0.1\textwidth]{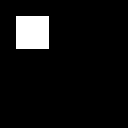}}
    \put(2.2,0){\includegraphics[width=0.1\textwidth]{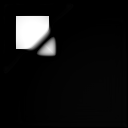}}
    \put(4.4,0){\includegraphics[width=0.1\textwidth]{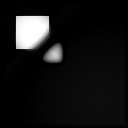}}
    \put(6.6,0){\includegraphics[width=0.1\textwidth]{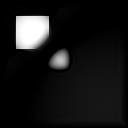}}
    \put(8.8, 0){\includegraphics[width=0.1\textwidth]{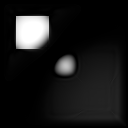}}
    \put(11,0){\includegraphics[width=0.1\textwidth]{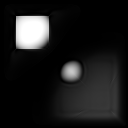}}
    \put(13.2,0){\includegraphics[width=0.1\textwidth]{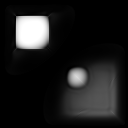}}
    \put(15.4,0){\includegraphics[width=0.1\textwidth]{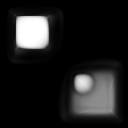}}
    \put(17.6,0){\includegraphics[width=0.1\textwidth]{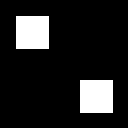}}
  \end{picture}
  \end{minipage}
}
\resizebox{1.0\linewidth}{!}{
  \begin{minipage}[h]{1.0\textwidth}
  \begin{picture}(20,2.4)
    \put(0,0){\includegraphics[width=0.1\textwidth]{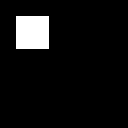}}
    \put(2.2,0){\includegraphics[width=0.1\textwidth]{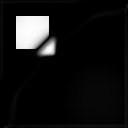}}
    \put(4.4,0){\includegraphics[width=0.1\textwidth]{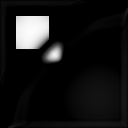}}
    \put(6.6,0){\includegraphics[width=0.1\textwidth]{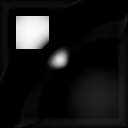}}
    \put(8.8, 0){\includegraphics[width=0.1\textwidth]{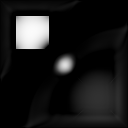}}
    \put(11,0){\includegraphics[width=0.1\textwidth]{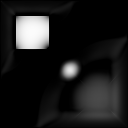}}
    \put(13.2,0){\includegraphics[width=0.1\textwidth]{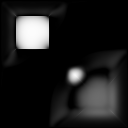}}
    \put(15.4,0){\includegraphics[width=0.1\textwidth]{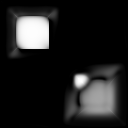}}
    \put(17.6,0){\includegraphics[width=0.1\textwidth]{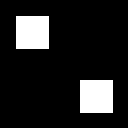}}
  \end{picture}
  \end{minipage}
}
\resizebox{1.0\linewidth}{!}{
  \begin{minipage}[h]{1.0\textwidth}
  \begin{picture}(20,2.4)
    \put(0,0){\includegraphics[width=0.1\textwidth]{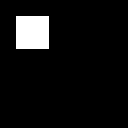}}
    \put(2.2,0){\includegraphics[width=0.1\textwidth]{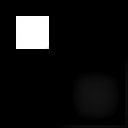}}
    \put(4.4,0){\includegraphics[width=0.1\textwidth]{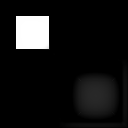}}
    \put(6.6,0){\includegraphics[width=0.1\textwidth]{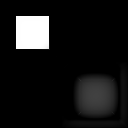}}
    \put(8.8, 0){\includegraphics[width=0.1\textwidth]{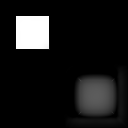}}
    \put(11,0){\includegraphics[width=0.1\textwidth]{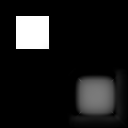}}
    \put(13.2,0){\includegraphics[width=0.1\textwidth]{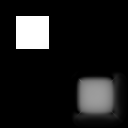}}
    \put(15.4,0){\includegraphics[width=0.1\textwidth]{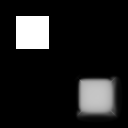}}
    \put(17.6,0){\includegraphics[width=0.1\textwidth]{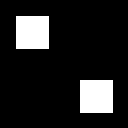}}
  \end{picture}
  \end{minipage}
}
\resizebox{1.0\linewidth}{!}{
  \begin{minipage}[h]{1.0\textwidth}
  \begin{picture}(20,2.4)
    \put(0,0){\includegraphics[width=0.1\textwidth]{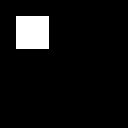}}
    \put(2.2,0){\includegraphics[width=0.1\textwidth]{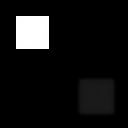}}
    \put(4.4,0){\includegraphics[width=0.1\textwidth]{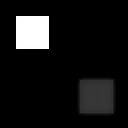}}
    \put(6.6,0){\includegraphics[width=0.1\textwidth]{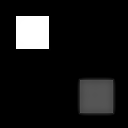}}
    \put(8.8, 0){\includegraphics[width=0.1\textwidth]{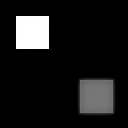}}
    \put(11,0){\includegraphics[width=0.1\textwidth]{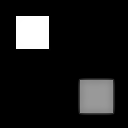}}
    \put(13.2,0){\includegraphics[width=0.1\textwidth]{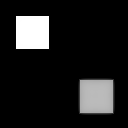}}
    \put(15.4,0){\includegraphics[width=0.1\textwidth]{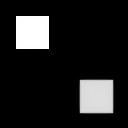}}
    \put(17.6,0){\includegraphics[width=0.1\textwidth]{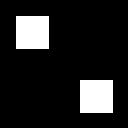}}
  \end{picture}
  \end{minipage}
}
\caption{Generalized Wasserstein geodesic between a scaled characteristic functions on a square and two differently scaled characteristic functions of squares of the same size. 
The dependance on the source term parameter is shown for the $L^2(H)$ model. From top to bottom 
$\penaltyPushforward = 10^{-2}, 10^{-1}, 10^{0}, 10^{1}$.}
\label{fig:effectDeltaSquare}
\end{figure}

\begin{figure}[!h]
\setlength{\unitlength}{0.05\textwidth}
\centering
\resizebox{1.0\linewidth}{!}{
  \begin{minipage}[h]{1.0\textwidth}
  \begin{picture}(20,2.4)
    \put(0,0){\includegraphics[width=0.1\textwidth]{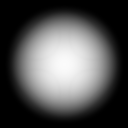}}
    \put(2.2,0){\includegraphics[width=0.1\textwidth]{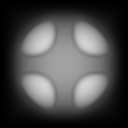}}
    \put(4.4,0){\includegraphics[width=0.1\textwidth]{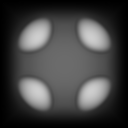}}
    \put(6.6,0){\includegraphics[width=0.1\textwidth]{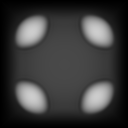}}
    \put(8.8, 0){\includegraphics[width=0.1\textwidth]{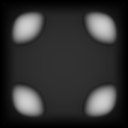}}
    \put(11,0){\includegraphics[width=0.1\textwidth]{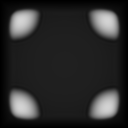}}
    \put(13.2,0){\includegraphics[width=0.1\textwidth]{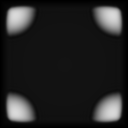}}
    \put(15.4,0){\includegraphics[width=0.1\textwidth]{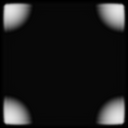}}
    \put(17.6,0){\includegraphics[width=0.1\textwidth]{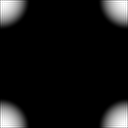}}
  \end{picture}
  \end{minipage}
}
\resizebox{1.0\linewidth}{!}{
  \begin{minipage}[h]{1.0\textwidth}
  \begin{picture}(20,2.4)
    \put(0,0){\includegraphics[width=0.1\textwidth]{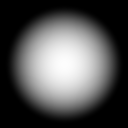}}
    \put(2.2,0){\includegraphics[width=0.1\textwidth]{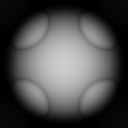}}
    \put(4.4,0){\includegraphics[width=0.1\textwidth]{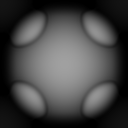}}
    \put(6.6,0){\includegraphics[width=0.1\textwidth]{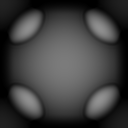}}
    \put(8.8, 0){\includegraphics[width=0.1\textwidth]{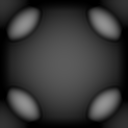}}
    \put(11,0){\includegraphics[width=0.1\textwidth]{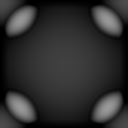}}
    \put(13.2,0){\includegraphics[width=0.1\textwidth]{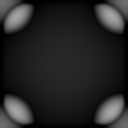}}
    \put(15.4,0){\includegraphics[width=0.1\textwidth]{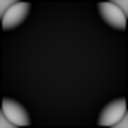}}
    \put(17.6,0){\includegraphics[width=0.1\textwidth]{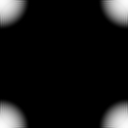}}
  \end{picture}
  \end{minipage}
}
\caption{Generalized Wasserstein geodesic connecting two translated bump functions are computed. The bump functions are periodically extended and 
$\penaltyPushforward = 10^{-2}$ (top) and $\penaltyPushforward = 10^{0}$ (bottom).}
\label{fig:SplittingBumpFunctions}
\end{figure}

In Figure~\ref{fig:3Circles} another type of interaction between generation and transport of mass is shown. The initial images at time $t=0$ consists of three scaled characteristic functions of balls, where one of this balls has smaller density value. The final image at time $t=1$ is based on the identical geometric configuration, but with swapped  densities.
Depending on the parameter $\penaltyPushforward$ a certain amount of mass is transported from the two balls with higher intensity in the image at time $t=0$. 
At the same time a blending of the transported masses as a compensation for the non balanced total mass can be observed. Figure~\ref{fig:3CirclesSource} shows plots of the functions $t\mapsto \int_D |z(t,\cdot)| \d x$, $t\mapsto \int_D z^+(t,\cdot) \d x$, and 
$t\mapsto \int_D  |z^-(t,\cdot)| \d x$ for the different values of $\delta$.

A striking observation in Figure~\ref{fig:squareMeasureSource} and 
Figure~\ref{fig:3CirclesSource} ist that 
$t\mapsto \int_D |z(t,\cdot)| \d x$ is approximately constant in time for the $L^2(H)$ model. 
This is in constrast to the $L^2(L^2)$ model as indicated in 
Figure~\ref{fig:squareMeasureSource}.

\begin{figure}[!h]
\setlength{\unitlength}{0.05\textwidth}
\centering
\resizebox{1.0\linewidth}{!}{
  \begin{minipage}[h]{1.0\textwidth}
  \begin{picture}(20,2.4)
    \put(0,0){\includegraphics[width=0.1\textwidth]{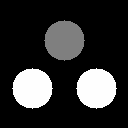}}
    \put(2.2,0){\includegraphics[width=0.1\textwidth]{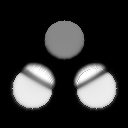}}
    \put(4.4,0){\includegraphics[width=0.1\textwidth]{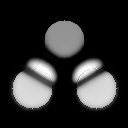}}
    \put(6.6,0){\includegraphics[width=0.1\textwidth]{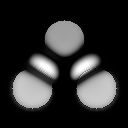}}
    \put(8.8, 0){\includegraphics[width=0.1\textwidth]{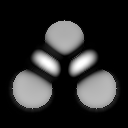}}
    \put(11,0){\includegraphics[width=0.1\textwidth]{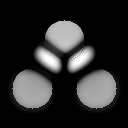}}
    \put(13.2,0){\includegraphics[width=0.1\textwidth]{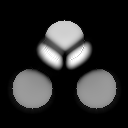}}
    \put(15.4,0){\includegraphics[width=0.1\textwidth]{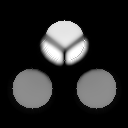}}
    \put(17.6,0){\includegraphics[width=0.1\textwidth]{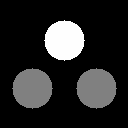}}
  \end{picture}
  \end{minipage}
}
\resizebox{1.0\linewidth}{!}{
  \begin{minipage}[h]{1.0\textwidth}
  \begin{picture}(20,2.4)
    \put(2.2,0){\includegraphics[width=0.1\textwidth]{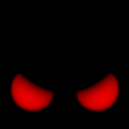}}
    \put(4.4,0){\includegraphics[width=0.1\textwidth]{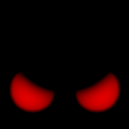}}
    \put(6.6,0){\includegraphics[width=0.1\textwidth]{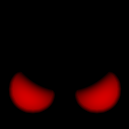}}
    \put(8.8, 0){\includegraphics[width=0.1\textwidth]{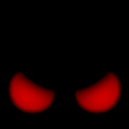}}
    \put(11,0){\includegraphics[width=0.1\textwidth]{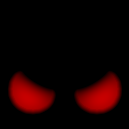}}
    \put(13.2,0){\includegraphics[width=0.1\textwidth]{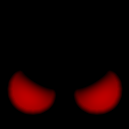}}
    \put(15.4,0){\includegraphics[width=0.1\textwidth]{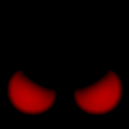}}
    \put(17.6,0){\includegraphics[width=0.1\textwidth]{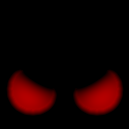}}
  \end{picture}
  \end{minipage}
}
\resizebox{1.0\linewidth}{!}{
  \begin{minipage}[h]{1.0\textwidth}
  \begin{picture}(20,2.4)
    \put(0,0){\includegraphics[width=0.1\textwidth]{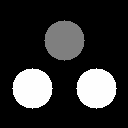}}
    \put(2.2,0){\includegraphics[width=0.1\textwidth]{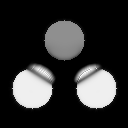}}
    \put(4.4,0){\includegraphics[width=0.1\textwidth]{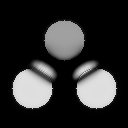}}
    \put(6.6,0){\includegraphics[width=0.1\textwidth]{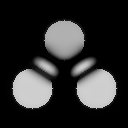}}
    \put(8.8, 0){\includegraphics[width=0.1\textwidth]{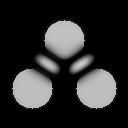}}
    \put(11,0){\includegraphics[width=0.1\textwidth]{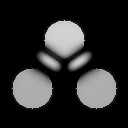}}
    \put(13.2,0){\includegraphics[width=0.1\textwidth]{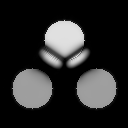}}
    \put(15.4,0){\includegraphics[width=0.1\textwidth]{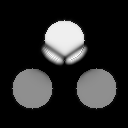}}
    \put(17.6,0){\includegraphics[width=0.1\textwidth]{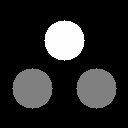}}
  \end{picture}
  \end{minipage}
}
\resizebox{1.0\linewidth}{!}{
  \begin{minipage}[h]{1.0\textwidth}
  \begin{picture}(20,2.4)
    \put(2.2,0){\includegraphics[width=0.1\textwidth]{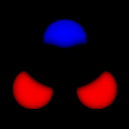}}
    \put(4.4,0){\includegraphics[width=0.1\textwidth]{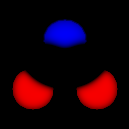}}
    \put(6.6,0){\includegraphics[width=0.1\textwidth]{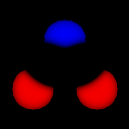}}
    \put(8.8, 0){\includegraphics[width=0.1\textwidth]{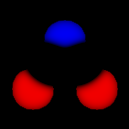}}
    \put(11,0){\includegraphics[width=0.1\textwidth]{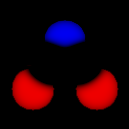}}
    \put(13.2,0){\includegraphics[width=0.1\textwidth]{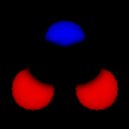}}
    \put(15.4,0){\includegraphics[width=0.1\textwidth]{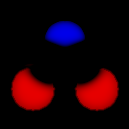}}
    \put(17.6,0){\includegraphics[width=0.1\textwidth]{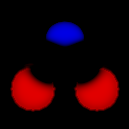}}
  \end{picture}
  \end{minipage}
}
\resizebox{1.0\linewidth}{!}{
  \begin{minipage}[h]{1.0\textwidth}
  \begin{picture}(20,2.4)
    \put(0,0){\includegraphics[width=0.1\textwidth]{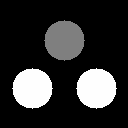}}
    \put(2.2,0){\includegraphics[width=0.1\textwidth]{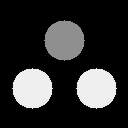}}
    \put(4.4,0){\includegraphics[width=0.1\textwidth]{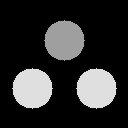}}
    \put(6.6,0){\includegraphics[width=0.1\textwidth]{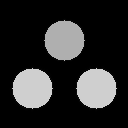}}
    \put(8.8, 0){\includegraphics[width=0.1\textwidth]{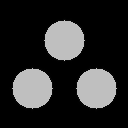}}
    \put(11,0){\includegraphics[width=0.1\textwidth]{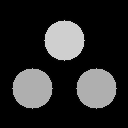}}
    \put(13.2,0){\includegraphics[width=0.1\textwidth]{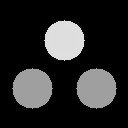}}
    \put(15.4,0){\includegraphics[width=0.1\textwidth]{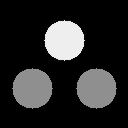}}
    \put(17.6,0){\includegraphics[width=0.1\textwidth]{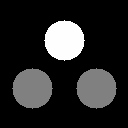}}
  \end{picture}
  \end{minipage}
}
\resizebox{1.0\linewidth}{!}{
  \begin{minipage}[h]{1.0\textwidth}
  \begin{picture}(20,2.4)
    \put(2.2,0){\includegraphics[width=0.1\textwidth]{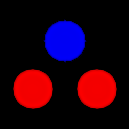}}
    \put(4.4,0){\includegraphics[width=0.1\textwidth]{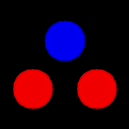}}
    \put(6.6,0){\includegraphics[width=0.1\textwidth]{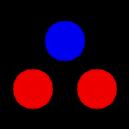}}
    \put(8.8, 0){\includegraphics[width=0.1\textwidth]{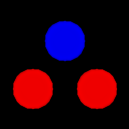}}
    \put(11,0){\includegraphics[width=0.1\textwidth]{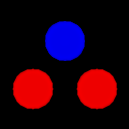}}
    \put(13.2,0){\includegraphics[width=0.1\textwidth]{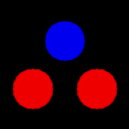}}
    \put(15.4,0){\includegraphics[width=0.1\textwidth]{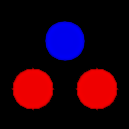}}
    \put(17.6,0){\includegraphics[width=0.1\textwidth]{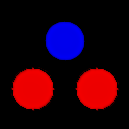}}
  \end{picture}
  \end{minipage}
}
\caption{Optimal transport geodesic between three scaled characteristic functions of balls with different densities. Here the source term parameter are $\penaltyPushforward = 10^{0}, 10^1, 10^2$. 
Bottom row: The distribution of the $L^1$-norm of the source term in time. }
\label{fig:3Circles}
\end{figure}

\begin{figure}
\hspace*{-5ex}
\resizebox{1.0\linewidth}{!}{
\begin{tikzpicture}
   \begin{axis}
     [ width=0.3\linewidth,
      at={(0.1\linewidth,0)},
      xmin=0.0, xmax=1.05,
      ymin=0, ymax=780,
      every axis/.append style={font=\footnotesize},
      axis lines=middle,
      axis line style={->},
       label style = {font=\footnotesize},
      xlabel={$t$},
      ylabel={\Umbruch{$\int_D |z(t, \cdot )| \d x$ \color{blue}{$\int_D z^{+}(t, \cdot ) \d x$} \color{red}{$\int_D z^{-}(t, \cdot ) \d x$ } } },
      y label style={at={(-1.0,1.0)}},
      yticklabels={, , , , , , },
      xtick={0.25,0.5,0.75,1},
     ]
     \addplot [thick] table[x expr=\thisrowno{0} / 128, y expr=\thisrowno{1} ] {images/OptimalTransportGeodesic_3Circles/Huber_eps1-4_del1/L1NormSourceTerm.txt};
     \addplot [thick, color=blue] table[x expr=\thisrowno{0} / 128, y expr=\thisrowno{1} ] {images/OptimalTransportGeodesic_3Circles/Huber_eps1-4_del1/L1NormSourceTermPos.txt};
     \addplot [thick, color=red] table[x expr=\thisrowno{0} / 128, y expr=\thisrowno{1} ] {images/OptimalTransportGeodesic_3Circles/Huber_eps1-4_del1/L1NormSourceTermNeg.txt};
    \end{axis}
\end{tikzpicture}
\begin{tikzpicture}
   \begin{axis}
     [ width=0.3\linewidth,
      at={(0.4\linewidth,0)},
        xmin=0.0, xmax=1.05,
        ymin=0, ymax=175,
        yticklabels={, , , , , , },
        axis line style={->},
        axis lines=middle,
      label style = {font=\footnotesize},
      xlabel={$t$},
      xtick={0.25,0.5,0.75,1},
       every axis/.append style={font=\footnotesize},
     ]
     \addplot [thick] table[x expr=\thisrowno{0} / 128, y expr=\thisrowno{1} ] {images/OptimalTransportGeodesic_3Circles/Huber_eps1-4_del10/L1NormSourceTerm.txt};
     \addplot [thick, color=blue] table[x expr=\thisrowno{0} / 128, y expr=\thisrowno{1}  ] {images/OptimalTransportGeodesic_3Circles/Huber_eps1-4_del10/L1NormSourceTermPos.txt};
     \addplot [thick, color=red] table[x expr=\thisrowno{0} / 128, y expr=\thisrowno{1}  ] {images/OptimalTransportGeodesic_3Circles/Huber_eps1-4_del10/L1NormSourceTermNeg.txt};
    \end{axis}
\end{tikzpicture}
\begin{tikzpicture}
   \begin{axis}
     [ width=0.3\linewidth,
      at={(0.7\linewidth,0)},
        xmin=0.0, xmax=1.05,
        ymin=0, ymax=27,
        yticklabels={, , , , , , },
        axis line style={->},
        axis lines=middle,
      label style = {font=\footnotesize},
      xlabel={$t$},
      xtick={0.25,0.5,0.75,1},
       every axis/.append style={font=\footnotesize},
     ]
     \addplot [thick] table[x expr=\thisrowno{0} / 128, y expr=\thisrowno{1}  ] {images/OptimalTransportGeodesic_3Circles/Huber_eps1-4_del100/L1NormSourceTerm.txt};
     \addplot [thick, color=blue] table[x expr=\thisrowno{0} / 128, y expr=\thisrowno{1}  ] {images/OptimalTransportGeodesic_3Circles/Huber_eps1-4_del100/L1NormSourceTermPos.txt};
     \addplot [thick, color=red] table[x expr=\thisrowno{0} / 128, y expr=\thisrowno{1}  ] {images/OptimalTransportGeodesic_3Circles/Huber_eps1-4_del100/L1NormSourceTermNeg.txt};
    \end{axis}
\end{tikzpicture}
}
\caption{Distribution of the $L^1$-norm of the source term in time for the example in Figure~\ref{fig:3Circles}. }
\label{fig:3CirclesSource}
\end{figure}
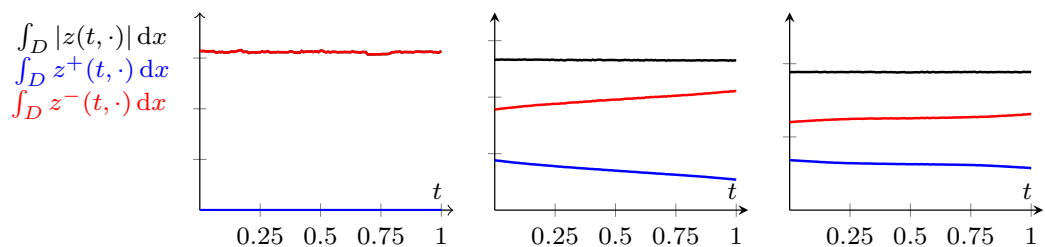

Finally in Figure~\ref{fig:WoodTextures} depicts an example for textured images. A generalized Wasserstein geodesic in case of the $L^2(H)$ cost functional is shown.

\begin{figure}[!h]
\setlength{\unitlength}{0.05\textwidth}
\centering
\resizebox{1.0\linewidth}{!}{
  \begin{minipage}[h]{1.0\textwidth}
  \begin{picture}(20,2.4)
    \put(0,0){\includegraphics[width=0.1\textwidth]{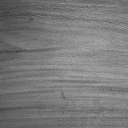}}
    \put(2.2,0){\includegraphics[width=0.1\textwidth]{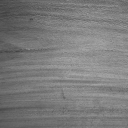}}
    \put(4.4,0){\includegraphics[width=0.1\textwidth]{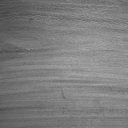}}
    \put(6.6,0){\includegraphics[width=0.1\textwidth]{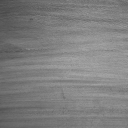}}
    \put(8.8, 0){\includegraphics[width=0.1\textwidth]{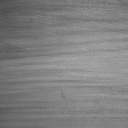}}
    \put(11,0){\includegraphics[width=0.1\textwidth]{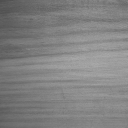}}
    \put(13.2,0){\includegraphics[width=0.1\textwidth]{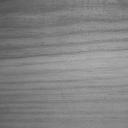}}
    \put(15.4,0){\includegraphics[width=0.1\textwidth]{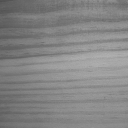}}
    \put(17.6,0){\includegraphics[width=0.1\textwidth]{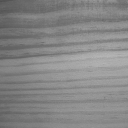}}
  \end{picture}
  \end{minipage}
}
\resizebox{1.0\linewidth}{!}{
  \begin{minipage}[h]{1.0\textwidth}
  \begin{picture}(20,2.4)
    \put(0,0){\includegraphics[width=0.1\textwidth]{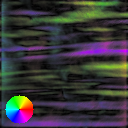}}
    \put(2.2,0){\includegraphics[width=0.1\textwidth]{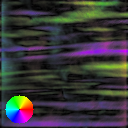}}
    \put(4.4,0){\includegraphics[width=0.1\textwidth]{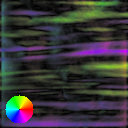}}
    \put(6.6,0){\includegraphics[width=0.1\textwidth]{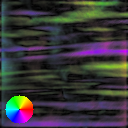}}
    \put(8.8, 0){\includegraphics[width=0.1\textwidth]{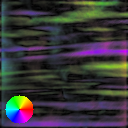}}
    \put(11,0){\includegraphics[width=0.1\textwidth]{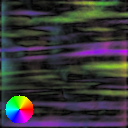}}
    \put(13.2,0){\includegraphics[width=0.1\textwidth]{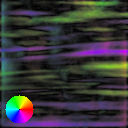}}
    \put(15.4,0){\includegraphics[width=0.1\textwidth]{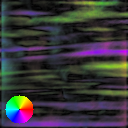}}
    \put(17.6,0){\includegraphics[width=0.1\textwidth]{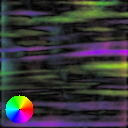}}
  \end{picture}
  \end{minipage}
}
\resizebox{1.0\linewidth}{!}{
  \begin{minipage}[h]{1.0\textwidth}
  \begin{picture}(20,2.4)
    \put(2.2,0){\includegraphics[width=0.1\textwidth]{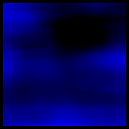}}
    \put(4.4,0){\includegraphics[width=0.1\textwidth]{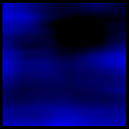}}
    \put(6.6,0){\includegraphics[width=0.1\textwidth]{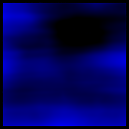}}
    \put(8.8, 0){\includegraphics[width=0.1\textwidth]{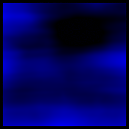}}
    \put(11,0){\includegraphics[width=0.1\textwidth]{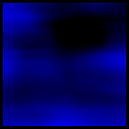}}
    \put(13.2,0){\includegraphics[width=0.1\textwidth]{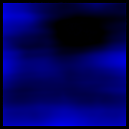}}
    \put(15.4,0){\includegraphics[width=0.1\textwidth]{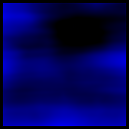}}
    \put(17.6,0){\includegraphics[width=0.1\textwidth]{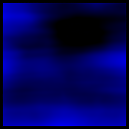}}
  \end{picture}
  \end{minipage}
}
\caption{Optimal transport geodesic between wood textures (top). Here the source term parameter is $\penaltyPushforward = 10^{0}$. Further the corresponding momentum (middle) and source term (bottom) are depited.}
\label{fig:WoodTextures}
\end{figure}

\paragraph*{Acknowledgements.} 
We acknowledge support by the German Science Foundation via the CRC 1060.

\end{document}